\documentclass[11pt]{article}
\usepackage[mathscr]{euscript}

\usepackage{amsmath,amssymb}
\usepackage{graphicx}
\usepackage{bbm}
\usepackage{xcolor}
\usepackage{stmaryrd}
\usepackage{upgreek}
\usepackage{esint}

 \newtheorem{remark}{Remark}
\newtheorem{theorem}{Theorem} \newtheorem{lemma}{Lemma}

\newtheorem{proposition}{Proposition}
 \newtheorem{corollary}{Corollary}

\def \R{{\mathbb{R}}}

\def \Z{{\mathbb{Z}}}

\def \T{{\mathbb{T}}}

\def \del {\partial}

\def \ue {{u^\eta}}
\def \Pe {P^\eta}
\def \pe {{p^\eta}}
\def \Pbe {{P_b^\eta}}
\def \Pei {{P_i^\eta}}

\title{$C^{0,\alpha}$ boundary regularity for the pressure in weak solutions of the $2d$ Euler equations}  \author{Claude W.  Bardos  \footnotemark[1] \and Edriss S. Titi \footnotemark[2]}

\begin{document}

\maketitle



\renewcommand{\thefootnote}{\fnsymbol{footnote}}

\footnotetext[1]{
Laboratoire J.-L. Lions, BP187, 75252 Paris Cedex 05, France. {
claude.bardos@gmail.com }}

\footnotetext[2]{ Department of Mathematics,
                 Texas A\&M University,
                 College Station, TX 77843-3368, USA; Department of Applied Mathematics and Theoretical Physics, University of Cambridge, Cambridge CB3 0WA UK; also Department of Computer Science and Applied Mathematics, Weizmann Institute of Science, Rehovot 76100, Israel.  {titi@math.tamu.edu} \quad {Edriss.Titi@damtp.cam.ac.uk}}

\maketitle

\begin{abstract}
The purpose of this note is to give a complete proof of a $C^{0,\alpha}$
regularity result for the pressure for weak solutions of the two-dimensional
``incompressible Euler equations'' when the fluid velocity enjoys the same type of regularity in a compact simply connected domain with $C^2$  boundary. To accomplish our result  we realize that it is compulsory to introduce a new weak formulation for the boundary condition of the pressure which is consistent with, and equivalent to, that of classical solutions.
\end{abstract}

\section{Introduction}
This contribution is devoted to the analysis of the regularity of the pressure, $p\,,$ associated with the weak  solutions:
\begin{equation}
(x,t) \mapsto u(x,t) \in C ([0,T]; C^{0,\alpha} (\Omega))\,, \quad \hbox{with} \quad \alpha \in (0,1)\,,
\end{equation}
for the Euler equations of incompressible invsicid/ideal fluid
\begin{equation}
\begin{aligned}
&\del_t u +\sum_{j=1}^d   \del_j (u_j u ) +\nabla p =0\,, \\
& \nabla\cdot  u=0\,, \quad \hbox{in} \,\, \Omega\,, \quad \hbox{and} \quad u\cdot \vec  n =0  \quad \hbox{on } \del \Omega\,, \label{euler1}
\end{aligned}
\end{equation}
where $\Omega \subset \R^d$ is a simply connected bounded domain with a smooth (say $C^2$) boundary, while
$\vec n( x)$   denotes   the extension in a neighborhood of the boundary $\del\Omega$ of
 the interior normal to the boundary. The tensorial  notation $v\otimes w$  of two vectors $v,w \in \R^d$ will be used whenever convenient for the matrix with entries $(v_i w_j)_{i,j=1}^d$ (and its various avatars), in particular in the next formula (\ref{wi}). Moreover, for $d\times d$ square matrices $A,B$ we denote by $\displaystyle A:B= {
 \rm{trace}} \big(A \cdot B^T \big)= \sum_{i,j=1}^d A_{i,j} B_{j,i}$\,.

In a compact domain with no boundary,  (basically the torus $\Omega=(\R/(L\Z))^d$) using  the divergence free condition, one deduces from (\ref{euler1}) the equation
\begin{equation}
-\Delta p= (\nabla_x\otimes\nabla_x) : (u\otimes u)\,, \label{wi}
\end{equation}
which uniquely determines the pressure (up to a constant) and also determines   its regularity in term of the regularity of the tensor $(u\otimes u)$. For instance, for $u\in C^{k,\alpha}(\Omega)\,,$ standard H\"older elliptic regularity (cf. \cite{Krylov} chapter 3 or \cite{LU} chapter 3) states that  for any integer $k\ge 0\,,$ $p$ belongs also to the space $C^{k,\alpha}(\Omega)\,.$ In the presence a boundary $\del\Omega\,,$ taking the scalar product with the interior normal $\vec n$ of the first equation of (\ref{euler1}) one obtains, now for $k\ge 2$, the relation:
\begin{equation}
-\del_{\vec n} p = \sum_{i,j=1}^d \del_j(u_iu_j) \vec n_i = (u\otimes u): \nabla \vec n\,, \quad \hbox{on}\,\, \del\Omega \label{wb}\,.
\end{equation}
In (\ref{wb}) the second identity follows from the fact that $u$ is tangential to the boundary $\del\Omega$, (i.e., $ u \cdot \vec n=0\,$ on $\del\Omega$ which is a level surface of the scalar function $u\cdot\vec n$). On  $\del \Omega\,,$ $\nabla \vec n$ is called the Weingarten matrix. It is determined in term of the principal curvatures of $\del \Omega\,,$  and in two dimensions it is a scalar, i.e., the curvature $\gamma$ of $\del\Omega\,.$ Therefore, it is natural to use
\begin{equation}
-\del_{\vec n} p =  (u\otimes u): \nabla \vec n \label{wb1}\,, \quad \hbox{on}\,\, \del\Omega\,,
\end{equation}
as the boundary condition for the pressure in weak formulation, when the boundary $\del\Omega\ \in C^2$\,.

The  first observation is that   equations (\ref{wi}) and (\ref{wb}) may define $p$ only up to a constant as stated in the following:
\begin{proposition}\label{unicity}
 Let $(p,q) \in (\mathcal D'(\overline{\Omega}))^2$ be two extendable distributional  solutions of (\ref{wi}) and (\ref{wb}), i.e. distributions which are defined in an open neighborhood of $\overline{\Omega}$. Then $r=p-q$ which is a solution of
 $
 -\Delta r=0
 $ in $\Omega$ and $\del_{\vec n} r=0$ on $\del\Omega$  is a constant . Hence there is at most one solution of the system (\ref{wi}) and (\ref{wb}) with the extra condition
 \begin{equation*}
 \int_{\Omega} p(x)dx =0\,.
 \end{equation*}
 \end{proposition}
 The second observation is that the existence and the regularity of a solution of   (\ref{wi}) and (\ref{wb})  follows also, for $k\ge 2\,,$ from the classical H\"older elliptic regularity for boundary value problems as described in chapter 6 of \cite{Krylov} and in the chapter 3 of \cite{LU}.
 This yields the existence and uniqueness  of  the pressure as stated in the following:
 \begin{theorem} \label{Grosbasic}
   Let $\del \Omega \in C^k$, and let $u\in C^{k,\alpha} ( \Omega)$ with $k\ge 2\,,$ be a divergence free vector field which is  tangential to the boundary. Then  there is one and  only one solution of (\ref{wi})  and  (\ref{wb}) :
   \begin{equation*}
   \begin{aligned}
 &\hbox{in} \quad \Omega \quad -\Delta p =\nabla\otimes \nabla : (u\otimes u)\,, 
 \\
  &\hbox{on }  \del \Omega   \quad -\del_{\vec n} p =(u\otimes u):\nabla  \vec n
 \quad \hbox{ and} \quad
 \int_{\Omega } p(x)dx=0\,.
\end{aligned}
\end{equation*}
  Moreover, this solution satisfies the estimate
  \begin{equation}
  \|p  \|_{C^{k,\alpha} ( \Omega)} \le C\|(u \otimes u) )\|_{C^{k,\alpha}(\Omega)}\,,  \label{goal}
 \end{equation}
 where $C$ is a positive constant that depends only on $\alpha$ and $\Omega$.
\end{theorem}
As briefly described in the conclusion section below the $C^{0,\alpha}$ regularity plays an  important role in the mathematical understanding of  turbulence, in particular in the presence of boundary effects.
Hence the purpose of the present contribution is to extend Theorem \ref{Grosbasic} above to the $C^{0,\alpha}(\Omega)$ case,  hence providing a detailed proof of a proposition already used in a previous article (cf. Proposition 2 in \cite{BT2}).
The expected result, which will be proven as Theorem \ref{Super} and Corollary \ref{super} in section \ref{Supersuper}, concerns this extension of Theorem \ref{Grosbasic} to the case $k=0\,,$ under the assumption that $\del \Omega \in C^2$.
 However, the weak boundary condition \eqref{wb1}  involves the quantity
 $\del_{\vec n} p$,  which might not be well defined on $\del\Omega$ in this case. Therefore, we propose an even weaker formulation than \eqref{wb1} for the boundary condition  which  involves the quantity
 $$
 \del_{\vec n} \Big(p + (u\cdot \vec \nabla_x d(x,\del\Omega))^2\Big)\,, \quad \hbox{on} \quad \del\Omega\,,
 $$
instead. This is motivated by the fact that for $u$ is smooth enough (say in $C^{0,\alpha}(\Omega)$ with $\alpha > \frac{1}{2}$) with $u\cdot\vec n=0$  on the boundary, $\del\Omega$\,, one has:
\begin{equation}\label{adjustment-term}
 \del_{\vec n} \Big(u\cdot \vec \nabla_x d(x,\del\Omega)\Big)^2 =0\,, \quad \hbox{on} \quad \del\Omega\,.
 \end{equation}
 Consequently, instead of the weak boundary condition \eqref{wb1} for the pressure we consider in this work the following    version
\begin{equation}
-\del_{\vec n} \Big( p + \big(u\cdot \vec \nabla_x d(x,\del\Omega)\big)^2\Big)=  (u\otimes u): \nabla \vec n \label{wb2}\,, \quad \hbox{on}\,\, \del\Omega\,.
\end{equation}
This is obviously equivalent to the boundary conditions (\ref{wb}) or  (\ref{wb1}) in the case of classical solutions; in particular, it is equivalent to (\ref{wb1}) when  $u\in C^{0,\alpha}(\Omega)$ with $\alpha > \frac{1}{2}$\,. However, when $\alpha \in (0,\frac{1}{2}]$\,, which included the Onsager's critical exponent $\alpha =\frac{1}{3}$\,, (\ref{wb2}) is a weaker formulation than (\ref{wb1}) for the boundary condition of pressure in the framework of weak solution to the Euler equations in the presence of a boundary. This is because the left-hand side of (\ref{wb2}) involves the sum of two terms, which we will show that it makes sense at the boundary, while each term might not necessarily be regular enough to make sense at the boundary on its own. It is this  boundary condition  that we will be adopting in this  contribution.

  In spite the fact that we strongly believe that the same type of result will hold in the three-dimensional case, which is a subject of future work,  we consider below only the two-dimensional case for the following reasons.

\begin{enumerate}

\item[(i)] Although the result seems to be very natural, the proof turned out to be more elaborated than expected. Therefore, we choose to consider a situation where we can provide the full details, while keeping the presentation user friendly.

\item[(ii)] We use a {\it global localisation near the boundary}, which may not be absolutely compulsory in the present case,  but as stated in the conclusion, this idea may be extremely useful for companion  problems where the analyticity properties have to be preserved.

\end{enumerate}

This work is organized as follows:

\begin{enumerate}

 \item As mentioned above  we focus on the  two-dimensional case and provide a  global representation of the neighborhood  of the boundary. This is done by introducing what is  called global  geodesic coordinates and then state our main results,  Theorem \ref{Super} and Corollary \ref{super}.
\item We introduce in section 3 an incompressible regularized family of vector fields, $u^\eta\in C^\infty (\overline{\Omega})$, which is tangential to the boundary $\del \Omega\,,$ and which converges  in the  $C^0(\overline{\Omega})$  norm to the velocity field $u\in C^{0,\alpha}(\Omega)$ as $\eta \to 0$. We then establish the $C^{0,\alpha}(\Omega)$ uniform estimate, with respect to $\eta\,,$  for the corresponding pressure $p^\eta$ of the regularized tensor $(\ue\otimes \ue)\,.$

\item The final result is obtained by letting $\eta \rightarrow 0\,.$
\item In section 4 we conclude by arguing  on the pertinence,  not only of this result,  but also of the method for  the progress of mathematical theory of turbulence with boundary effects.
 \end{enumerate}

 \section{Global geodesic coordinates near the boundary $\del \Omega\,.$}
As we have mentioned above, for sake of clarity and also with further applications in mind we focus on the  two-dimensional case. We start  with a  parametric representation of $\del\Omega\,,$ a closed $C^2$ curve of length  $L\,:$
 $$
 \theta \in \T= \R \slash (\Z L)  \mapsto  x(\theta) =(x_1(\theta),x_2(\theta) ) \in \del\Omega\,,
 $$
  with $\vec \tau(\theta) $ and $\vec n(\theta) $ being, respectively, the unit tangent and interior normal vectors at the boundary:
 \begin{equation*}
  \begin{aligned}
  & \vec \tau (\theta)=\vec \tau(x(\theta))  =(  x'_1(\theta),   x'_2(\theta) )\,, \,\, \vec n (\theta) =\vec n (x(\theta))=  (- x'_2 (\theta),  x'_1(\theta) )\\
&\hbox{ with}\,\, | \vec n (\theta)|^2=| \vec \tau  (\theta)|^2=( x'_1(\theta) )^2 +( x'_2(\theta) )^2 =1\,.
\end{aligned}
   \end{equation*}
 Let $d(x,\del \Omega)$ denote the distance of any point $x\in \R^2 $ to $\del \Omega$ . Then there exists a $ \delta >0$ such that on the open set
  $$
  V_{  \delta}=\{x\in\R^2\quad \hbox{with} \quad  d(x,\del\Omega)< {\delta}\}\,,
  $$
  there  is a unique point $\hat x(\theta)  \in \del \Omega$ with
 $d(x,\del\Omega) =|x-\hat x(\theta) )|\,.$ Then the mapping
   $x\mapsto \hat x(\theta) $ belongs to $  C^2( {V_\delta}, \del \Omega) \,,$ and  for $x\in V_\delta\,,$  one has the formula
   \begin{equation*}
   \nabla_x d(x,\del\Omega) =\vec n (\hat x(\theta))\,,
   \end{equation*}
   while, in the absence of confusion, the notations $\vec n(x)$ and $\vec\tau (x) $  will be used for $\vec n(\hat x(\theta)) $  and $\vec \tau (\hat x(\theta)) $, respectively.
   Observe that 									
\begin{equation*}
 \vec\tau\,'(\theta) \wedge \vec n\,'(\theta)=  x'_1 (\theta) x_1''(\theta)  + x'_2(\theta)  x_2''(\theta) =\frac{d}{ d \theta}| x'(\theta)|^2=0\,,
\end{equation*}
which implies the relation
\begin{equation}\label{normal-tangent-derivatives}
\vec n\,'(\theta) = \gamma(\theta) \vec \tau (\theta) \quad\hbox{and} \quad \vec \tau\,'(\theta)= \gamma(\theta) \vec n(\theta)\,,
\end{equation}
with
\begin{equation*}
\gamma(\theta)= x_1''(\theta) x_2'(\theta) -x'_1(\theta) x_2''(\theta)\,,
\end{equation*}
being the curvature of the boundary $\del \Omega \,.$ Therefore  the mapping:
\begin{equation*}
(\theta,s)\mapsto  X(s,\theta) =  x(\theta)  + s\vec n(  x(\theta))\,,
\end{equation*}
defines a global $C^2$ diffeomorphism of $[-\delta ,\delta ] \times (\R/(L\Z))$ onto $\overline{V_\delta}\,.$ Moreover,   for any vector map $x\in \overline \Omega \mapsto v(x)\,,$ as soon as $x\in\overline V_\delta\cap \overline{\Omega}\,,$ using the above notations, one has:
\begin{equation*}
v(x)= \big(v(x)\cdot\vec \tau(x)\big) \vec \tau(x) +\big(v(x)\cdot\vec n(x)\big) \vec n  (x) \,.
\end{equation*}
Below, for the sake of clarity,   the symbol $X$ is used for any $x=X(s,\theta)$, for $(s,\theta)\in [-\delta ,\delta ] \times (\R/(L\Z))$\,, and the following formulas due to this representation are recalled:
\begin{equation*}
\begin{aligned}
&\del_s X(s,\theta) =\vec n(\theta)
\,,\quad \del_\theta X(s,\theta) = J(s,\theta) \vec \tau (\theta)\,,\\
&\hbox{with}\quad  J(s,\theta)= 1 + s\gamma(\theta) >0  \quad \hbox{for}\quad |s|<\delta\,.
\end{aligned}
\end{equation*}
From the relation
\begin{equation*}
                     \begin{pmatrix}  \del_s X_1  & \del_\theta  X_1  \\
                                       \del_s X_2  & \del_\theta  X_2    \end{pmatrix}
                     \begin{pmatrix}  \del_{X_1} s  &   \del_{X_2} s  \\
                                      \del_{X_1} \theta   &   \del_{X_2} \theta \end{pmatrix}
                                      =
                     \begin{pmatrix}  1 &  0  \\
                                      0 &  1 \end{pmatrix}\,,
 \end{equation*}
one deduces the formula:
\begin{equation} \label{gradients}
\nabla_X\theta  =\frac {\vec \tau ( s,\theta))}{  J(s,\theta) } \quad \hbox{and}\quad
 \nabla_X s  = \vec n(   \theta ) \,.
\end{equation}

Moreover, for $v \in C^1$ and $q\in C^2$ defined within $V_\delta$ one has
\begin{subequations}
\begin{equation}
 \nabla_x \cdot v=\frac1J\Big(\del_s ( J(v\cdot \vec n)) +\del_\theta(v\cdot \vec \tau)\Big)  \,, \label{clas1}
 \end{equation}
  \begin{equation}
  \nabla_x \wedge v = \frac1J\Big(\del_s (J (v\cdot \vec \tau)) -\del_\theta(v\cdot \vec n)\Big) \,,\label{clas2}
  \end{equation}
 \begin{equation}
 \Delta_x q = \frac 1J \del_s (J\del_s q) + \frac1J \del_\theta (  \frac 1J\del_\theta q)  \,.\label{clas3}
 \end{equation}
 \end{subequations}


 \section{Application of the global geodesic coordinates to $C^{0,\alpha}$ weak solutions of  the boundary-value problem (\ref{wi}) and (\ref{wb2}).} \label{Supersuper}

Let $\delta>0$ be small enough, as specified  in section 2, and let $\epsilon \in (0,\delta)$ be given.  Let $\phi :[0,\infty) \mapsto[0,1]$ be a $C^\infty$ non-increasing function such that $\phi (s)=1$ for $s\in [0,\delta-\epsilon]$ and $\phi(s)=0$
for $s\ge \delta\,.$ We consider the function $\phi(d(x,\del\Omega))$ which will be also denoted by $\phi(x)$. Observe that $\phi(x)$ belongs to $C_{\rm{c}}^2(\mathbb{R}^2)$ since $\del \Omega \in C^2\,.$

 Next we  state the main results of this contribution.

    \begin{theorem}\label{Super}
     Let $u \in C^{0,\alpha}(\Omega)$ be a divergence free vector field which is tangential to the boundary $\del \Omega$. Then there exists a unique function  $P$ defined on $\Omega$ with the following properties:
    \begin{enumerate}
    \item $P$ belongs to the space $C^{0,\alpha}(\Omega)$
    and satisfies the estimate:
    \begin{equation*}
    \|P\|_{C^{0,\alpha}(\Omega)} \le C \|u\otimes u \|_{C^{0,\alpha}(\Omega)}\,,
    \end{equation*}
    with a constant $C$ which depends only on $\alpha$ and $\Omega\,.$
    \item Denote by $P(s,\theta)=P(X(s,\theta))$, for $(s,\theta)\in [0,\delta) \times (\R/(L\Z))$\,, then the map $s\mapsto \del_s P(s,\cdot)$  belongs to $C([0,\delta); H^{-2}(\R/(L\Z)))\,,$  which implies that $\del_{\vec n} P$ is well defined on $\del \Omega\,$ with values in $H^{-2}(\del \Omega)$.
 \item    $P$ solves the following boundary-value problem:
    \begin{subequations}
    \begin{equation}
    \hbox{on}\quad  \del\Omega \quad \del_{\vec n} P = \gamma (u\cdot \vec \tau)^2 \,,\label{wneumann}
    \end{equation}
    \begin{equation}
   \hbox{in} \quad \Omega \quad -\Delta P=\big(\nabla\otimes\nabla\big):\big(u\otimes u\big) -\Delta \big(\phi(x)\big(u(x)\cdot \vec n(x)\big)^2\big)
   \end{equation}
   \begin{equation}
  \hbox{and} \quad   \int_\Omega P(x)dx = \int_\Omega \phi(x) \big(u(x)\cdot \vec n(x)\big)^2dx\,.
    \end{equation}
    \end{subequations}
    \end{enumerate}
    \end{theorem}

Note that by using the global geodesic coordinates the right-hand side of \eqref{wb2} takes the form:
$$
(u\otimes u): \nabla \vec n =  \gamma (u\cdot \vec \tau)^2 \quad \hbox{on} \quad \del \Omega\,,
$$
which clarifies  the right-hand side of \eqref{wneumann}.

\begin{remark}
   As stated in the introduction, the $C^{0,\alpha}(\Omega) $ regularity  for the pressure, $p$, is not convenient enough to deduce,
   as it is usually  done in the case of classical solutions,  from the boundary condition $u\cdot \vec n=0$\,,  the left-hand side of relation \eqref{wb}, namely,
   \begin{equation}\label{traceclassic}
  \hbox{ on}\quad  \del\Omega  \quad  \del_{\vec n} p =\big(\nabla \cdot \big(u\otimes u\big)\big) \cdot \vec n\,.
  \end{equation}
The quantity $\nabla \cdot \big(u\otimes u\big)$  may lose any meaning on the boundary, because for $u\in C^{0,\alpha}\,$ it is  defined only in the sense of distribution. Therefore, since the boundary is $C^2$ one is tempted to use instead the right-hand side of relation \eqref{wb} for the boundary condition on the pressure, namely,
\begin{equation}\label{traceclassic1}
\hbox{ on}\quad  \del\Omega  \quad  \del_{\vec n} p = \gamma (u\cdot \vec \tau )^2\,,
\end{equation}
which is equivalent to the original boundary condition for classical solutions. However, we realised that for $u\in C^{0,\alpha}\,$ the term $\del_{\vec n} p$ might not make sense on its own at the boundary.
Alternatively, we have argued that the pressure, $p$,  should satisfy the boundary condition \eqref{wb2} instead, which is equivalent to \eqref{traceclassic1} when $u\in C^{0,\alpha}\,$ with $\alpha > \frac{1}{2}\,$, and  which is a genuine weak formulation of the boundary condition for the pressure in this case, as it is conspicuous from the statement of the next corollary.

%

  \end{remark}
   From the above theorem,  considering the function
  $
  p= P-\phi(x) (u\cdot \vec n )^2\,,
 $
 one deduces the following:

\begin{corollary} \label{super} Let $u \in C^{0,\alpha}(\Omega)$ be a divergence free vector field which is tangential to the boundary $\del \Omega$. Then there exists a unique function $p\in C^{0,\alpha}(\Omega)$ which is a
 solution of the boundary-value problem
 \begin{equation}
 \hbox{in} \quad \Omega \quad-\Delta p = \big( \nabla\otimes \nabla \big) : \big (u\otimes u\big ) \,, \quad \hbox{ and } \int_\Omega p(x)dx =0
 \end{equation}
 and satisfies    the boundary condition (\ref{wb2}), i.e.,
 \begin{equation}
 \hbox{on} \quad \del \Omega \quad \del_{\vec n} \big(p+ (u \cdot \vec n )^2\big) =\gamma (u\cdot \vec \tau )^2\,.
\end{equation}
Moreover,
\begin{equation}
\|p\|_{C^{0,\alpha}(\Omega)} \le C \|u\otimes u\|_{C^{0,\alpha}(\Omega)}\,,
\end{equation}
for some positive constant which depends only on $\alpha$ and $\Omega$\,.
 \end{corollary}

  The proofs of Theorem \ref{Super} and Corollary \ref{super} are organized as follows:

\begin{enumerate}

\item We start by constructing a  regularization $u^\eta\in C^{\infty}(\overline{\Omega})$ of the velocity vector field $u\in C^{0,\alpha}(\Omega)$, for $\eta >0$ small enough, and which converges in the $C^0(\overline{\Omega})$ norm to $u$ as $\eta \to 0$; moreover it also  satisfies the estimate $\|u^\eta\|_{C^{0,\alpha}(\Omega)}\le C \|u \|_{C^{0,\alpha}(\Omega)}$, for some positive constant $C$ which is independent of $\eta$ and $\alpha$. In particular, we require $u^\eta$ to be divergence free and tangential to the boundary. This in turn allows us to invoke Theorem \ref{Grosbasic} for the case when $k=2$ (with $u$ replaced by $u^\eta$) to obtain the corresponding regularized pressure $p^\eta\in C^{2,\alpha}(\Omega)$. Then we consider near the boundary  modification of the regularized pressure by introducing the $C^{2,\alpha}(\Omega)$ function
$$
P^\eta(x) =p^\eta(x) +\phi(x)\big(u^\eta(x)\cdot \vec n(x)\big)^2\,,
$$
for all $x\in \overline{\Omega}$. Note that in this classical context, and by virtue of  \eqref{adjustment-term}, one has
\begin{equation*}
 \hbox{on} \quad \del \Omega \quad \del_{\vec n}P^\eta = \del_{\vec n}p^\eta\,.
\end{equation*}

\item Next we decompose $P^\eta$ into two functions  $P_b^\eta$ and $P_i^\eta$, with overlapping supports, where the support of   $P_b^\eta$ is near the boundary of $\Omega$, and the support of $P_i^\eta$ is a compact subset in the   interior of $\Omega$.
\item Representing $\Pe$ in $V_\delta\cap \overline{\Omega}$, in terms of the global geodesic coordinates near the boundary, we then establish a ``trace" theorem in which we prove  the ``uniform continuity" with respect to $s \in [0,\delta]$, i.e., up to the boundary, of the function $\del_s P^\eta(s,\cdot)\,$ with values in $H^{-2}(\R/(L\Z))\,.$ Consequently we accomplish  the estimate
\begin{equation}\label{Full-terms-estimate}
\|P^\eta\|_{C^{0,\alpha}(\Omega)} \le \|P_b^\eta\|_{C^{0,\alpha}(\Omega)} + \|P_i^\eta\|_{C^{0,\alpha}(\Omega)} \le  C\|(u^\eta\otimes u^\eta)\|_{C^{0,\alpha}(\Omega)}+D\|P^\eta\|_{L^\infty(\Omega)}\,,
\end{equation}
with positive constants $C$ and $D$ that are independent of $\eta$, and which depend only on $\Omega$ and $\alpha$.
\item Taking  advantage of the fact that   the constants $C$ and $D$ in \eqref{Full-terms-estimate}  are independent of $\eta$ we can show that
$$
\frac{\|P^\eta\|_{L^\infty(\Omega)}}{\|(u^\eta\otimes u^\eta)\|_{C^{0,\alpha}(\Omega)}}
$$
remains bounded for small values of $\eta$. This allows us to replace the constant $D$ in \eqref{Full-terms-estimate}  by zero on the expense of a larger constant $C$.
Eventually, insisting on the fact that the constant $C$ in  \eqref{Full-terms-estimate} depend only on $\alpha$ and $\Omega$ and that $D=0$ one can let $\eta \to 0\,,$  which allows us  to complete the proof.
 \end{enumerate}
\subsection{Adequate regularization of the velocity field}\label{regu}

The regularization process is based on the following (classical):
\begin{lemma}\label{regularized-velocity}
Let  $u\in C^{0,\alpha} (\Omega)$ be a divergence free and tangential to the boundary vector field defined in a bounded simply connected domain $\Omega$ with $C^2$ boundary. Then,  there exists an approximation family  $u^\eta \in C^{\infty }(\overline{\Omega})$\, of divergence free vector fields which are tangential  to the boundary and which converges to $u$ in the $C^0(\overline{\Omega})$ norm as $\eta\rightarrow 0$\,.  Moreover,
\begin{equation}\label{Holder-bound-regularization}
\|u^\eta\|_{C^{0,\alpha} (\Omega)} \le C \|u\|_{C^{0,\alpha} (\Omega)}\,,
\end{equation}
for some positive constant $C$ which is independent of $\eta$ and $\alpha$.
\end{lemma}
\begin{remark}
By a compactness argument it follows from (\ref{Holder-bound-regularization}) that the  convergence also holds in the $C^{0,\beta}(\Omega)$ norm for any $\beta \in (0,\alpha)$ as $\eta \to 0$.
\end{remark}
\begin{proof} of the Lemma \ref{regularized-velocity}:
Let  $\Psi$ by the unique solution in $H^1_0(\Omega)$ of the elliptic boundary-value problem:
\begin{equation}
\hbox{in }\quad \Omega  \quad -\Delta \Psi = \nabla \wedge u \quad \hbox{and on}\quad  \del \Omega \quad \Psi =0\,,
\label{cartan1}
\end{equation}
where the equation holds in $H^{-1}(\Omega)$ and the boundary condition in the trace sense. Consider the vector field $v=u-\nabla\wedge \Psi$ which satisfies  the relations $\nabla\wedge v=0\,$ and $\nabla\cdot v=0\,$ in $\mathcal{D}'(\Omega)$; moreover,  $v\cdot n =0$ in $H^{-1/2}(\del \Omega)$. Therefore, $\Delta v =0$ in $\mathcal{D}'(\Omega)$, and consequently $v\in C^\infty(\Omega)$.  Therefore, since $\nabla\wedge v=0\,$ and $\Omega\,$ is simply connected we have $v=\nabla q$ for some
$q\in C^\infty(\Omega)\,.$ Since $\nabla\cdot v=0$ in $\Omega$ and $v\cdot\vec n=0$ on $\del \Omega$ one concludes:
\begin{equation*}
\hbox{in }\quad \Omega \quad -\Delta q=0\quad \hbox{and on}\quad  \del \Omega \quad \del_{\vec n} q=0
\end{equation*}
 which implies that $q$ is constant. Thus,  $\nabla^\bot \Psi= u\in C^{0,\alpha}(\Omega)\,$ which implies that $\Psi \in C^{1,\alpha} (\Omega)$.

Next, we recall from section \ref{Supersuper} the function $\phi(x)\in C_{\rm{c}}^2(\mathbb{R}^2)$  and  that   $\rm{supp}(\phi )\subset \overline{V_\delta}$.  We decompose
\begin{equation*}
\Psi = \Psi_b + \Psi_i := \phi \Psi + (1-\phi)\Psi\,.
\end{equation*}
Consider the mollifier
\begin{equation*}
\begin{aligned}
&\rho^\eta(x)= \frac{1}{\eta^2} \rho(\frac{ x}{\eta}) \quad
   \hbox{with}\quad \rho \in C_{\rm{c}}^\infty(\mathbb{R}^2)\, \quad \hbox{is a radial function} \\
   &\quad \rho(x)\ge 0\,, \quad \rm{supp}(\rho ) \subset \{|x| \le 1\} \quad  \hbox{and} \quad \int_{\R^2} \rho(x) dx=1\,.
 \end{aligned}
 \end{equation*}
 Since $\Psi_i \in C_{\rm{c}}^{1,\alpha}(\Omega)$ then for $\eta$ small enough the function
  \begin{equation*}
  \Psi_i^\eta=\rho_\eta \ast \Psi_i \in C_{\rm{c}}^\infty(\Omega)\,.
  \end{equation*}
 Moreover, $\Psi_i^\eta$ converges   in the $C^1(\overline{\Omega})$ norm to $\Psi_i^\eta$,  and
$\|\Psi_i^\eta\|_{C^{1,\alpha} (\Omega)} \le C \|\Psi\|_{C^{1,\alpha} (\Omega)}$\,,
with a positive constant $C$ which is independent of $\alpha$ and $\eta$.

  To prove the same result for $\Psi_b$ we use the global geodesic coordinates introduced above. Since the  mollifier $\rho(s,\theta)$ is a radial function then it is an even function with respect to the $s$ variable, i.e., $\rho (s,\theta)= \rho(-s,\theta)$. Next,  we consider the odd extension of  the function $\Psi_b(s,\theta)$ with respect the $s$ variable, namely, we define:
  \begin{equation*}
  \tilde{\Psi}_b(s,\theta) =\left\{\begin{aligned}  \Psi_b(s,\theta)  \,\, \, \mbox{if} &\,\,\,   s \ge 0\\
                                                    -\Psi_b(-s,\theta) \, \,\, \mbox{if}  & \,\, \,  s  \le 0
                                                     \end{aligned}   \right. \,.
                                                     \end{equation*}
  Observe that $\tilde{\Psi}_b \in C_{\rm{c}}^{1,\alpha} (\R \times (\R/(L\Z)))$  satisfying $\tilde{\Psi}_b(0,\theta)=0\,.$  As a consequence
$\tilde{\Psi}^\eta_b:= \rho_\eta\ast \tilde{\Psi}_b \in  C^\infty_{\rm{c}} (\R \times (\R/(L\Z)))$, satisfying $\tilde{\Psi}^\eta_b(0,\theta)=0\,.$ Moreover, $\tilde{\Psi}^\eta_b$ converges in the $C^{1} (\R \times (\R/(L\Z)))$ norm, and in particular in  $C^{1} (\overline{\Omega})$ norm, to $\tilde{\Psi}_b$ as $\eta\rightarrow 0$. In addition, one can easily see that $\|\tilde{\Psi}^\eta_b\|_{C^{1,\alpha} (\Omega)} \le C \|\Psi\|_{C^{1,\alpha} (\Omega)}$\,,
with a positive constant $C$ which is independent of $\alpha$ and $\eta$.

Taking $u^\eta = \nabla^\bot \Psi^\eta$ and combining the above arguments one can complete the proof.

\end{proof}

As a consequence of the above construction of $u^\eta$ we invoke Theorem \ref{Grosbasic}, for the case $k=2$, to show that there exists a unique solution $p^\eta\in C^{2,\alpha}(\Omega)\,$ of the boundary-value problem:
\begin{equation} \label{regu2}
\hbox{ in} \quad \Omega  \quad -\Delta \pe = \big(\nabla\otimes \nabla\big) :\big(\ue\otimes\ue\big) \,, \quad\hbox{ on}\quad \del\Omega  \quad  \del_{\vec n} \pe = \gamma(\ue\cdot\vec\tau)^2\,
\quad\hbox{ and }\quad \int_\Omega \pe(x) dx =0\,.
\end{equation}

\subsection{ Boundary and interior functions}
To establish the uniform, with respect to $\eta$, $C^{0,\alpha}$ regularity estimate  for the pressure $p^\eta$ it seems compulsory to introduce different treatment of $p^\eta$ in the interior of $\Omega$, away from the boundary,  and near the boundary $\del\Omega\,.$ Therefore, besides the numbers $ \delta>0$ and $\epsilon\in (0,\delta)$  used before in the construction of the global geodesic representation of the neighborhood $V_{ \delta}\,$ and the cut-off function $\phi(x)\,,$ we introduce the following  positive numbers satisfying:
$$
0<\delta_1<\delta_2-\epsilon<  \delta_3  <\delta-2\epsilon\,.
$$
Moreover, for $s\in [0,\infty)$ we introduced the following three functions $s\mapsto \phi (s)\,$ (defined earlier), $s\mapsto \phi_b(s)$ and $\phi_i(s)$  ($b$ stands for boundary and $i$ for interior) belonging to $C^\infty([0,\infty))\,$ with the following properties:
 \begin{equation*}
 \phi (s)=\left\{\begin{aligned}  1  \, \, \mbox{if} &\,\,  0\le s\le\delta-\epsilon\\
                                                    0\, \, \mbox{if}  & \,\,   s\ge \delta
                                                     \end{aligned}  \right.\,,
\quad \phi_i(s)=\left\{\begin{aligned}  0  \, \,\mbox{if} &\, \, 0\le  s\le \delta_1\\
                                                    1\,\,  \mbox{if}  & \, \,  s\ge\delta_2-\epsilon
                                                                                                           \end{aligned}  \right.
\quad \mbox{and}\quad \phi_b =\left\{\begin{aligned}  1 \,\, \mbox{if} &\,\,  0\le s<\delta_3 + \epsilon \\
                                                     0 \, \, \mbox{if}  & \,\,   s\ge \delta-\epsilon
                                    \end{aligned}  \right. \,,
\end{equation*}
where $\phi,\phi_b$ are non-increasing and $\phi_i$ is non-decreasing.

As before, with the absence of confusion, for $\delta$ small enough we denote by

\begin{equation*}
\phi  (x) = \phi (d(x,\del\Omega))\,, \phi _b(x) = \phi_b(d(x,\del\Omega))\quad \hbox{and} \quad \phi_i(x) = \phi_i (d(x,\del\Omega))
\end{equation*}
which are  $C^2(\overline \Omega)\,.$

With $u^\eta$ as in section \ref{regu} and $p^\eta$ the classical solution of the boundary-value problem (\ref{regu2}) we define the following functions:
\begin{equation} \label{depart} \begin{aligned}
&\Pe(x)= \pe(x) + \phi(x) (\ue(x)\cdot \vec n(x))^2 \,, \\
&\Pei(x)=\phi_i(x) \Pe(x) =\phi_i(x)( (\pe(x) + \phi(x) (\ue(x))\cdot \vec n(x))^2)\,,\, \\
&\Pbe(x)=\phi_b(x) \Pe(x) = \phi_b(x)((\pe(x)  + (\ue(x)\cdot \vec n(x))^2)\,,
\end{aligned}
\end{equation}
where we used above  the relation $\phi_b(x)\phi(x)=\phi_b(x)$.

 \begin{center}
 \begin{figure}
 \includegraphics[width=13cm]{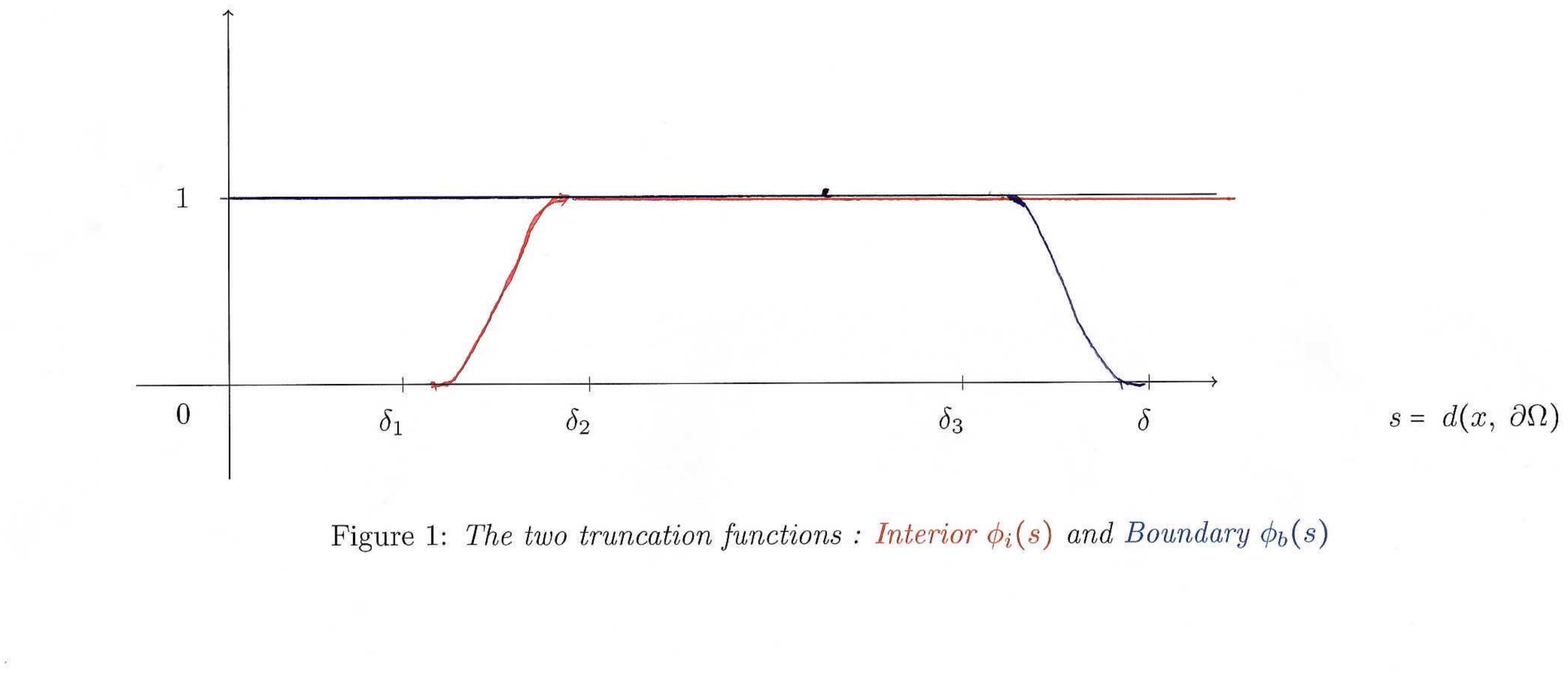}
\end{figure}
 \end{center}

\subsection{ Uniform estimates for $\Pe$}
In the next two sections we establish uniform estimates in $\eta$ for $\Pe$. To this end we take advantage of the above overlapping decomposition of  $\Pe$ into $\Pei$ and $\Pbe$. A first estimate comes directly from the definition of $\Pei$ and this is the objective of:
\begin{proposition} \label{Pei} The function $\Pei$ defined by (\ref{depart}) satisfies the estimate
\begin{equation}
\|\Pei\|_{C^{0,\alpha}(\Omega)} \le C_i\|u^\eta\otimes u^\eta \|_{C^{0,\alpha}(\Omega)} +D_i  \|\Pe \|_{L^\infty (\Omega)}\,, \label{Piii}
\end{equation}
with positive constants $C_i$ and $D_i$ which depend only on $\alpha$ and $\Omega\,$ and in particular they  are independent of $\eta\,.$
\end{proposition}
\begin{proof}
From  \eqref{regu2} and \eqref{depart} we observe  that $\Pei$ is a classical solution of the equation
   \begin{equation}
   \begin{aligned}
& -\Delta\Pei= \phi_i(x) \Bigg( (\nabla_x\otimes \nabla_x ): (u^\eta\otimes u^\eta) -\Delta_x \big(\phi(x) (\ue(x)\cdot \vec{n}(x))^2\big)\Bigg) \\
&\quad \quad\quad\quad \quad\quad  \quad\quad -2(\nabla_x \phi_i)\cdot \nabla_x \Pe -(\Delta_x  \phi_i) \Pe \,, \label{Piiii}
 \end{aligned}
 \end{equation}
 where  in both sides are functions with compact support in $\Omega$. Hence the solution $\Pei$ of \eqref{Piiii} is given by the fundamental formula

  \begin{equation*}
  \begin{aligned}
& \Pei =   \frac1{2\pi} \log\frac1{|x|}\ast \Bigg(\phi_i(x) \Bigg( (\nabla_x\otimes \nabla_x ): (u^\eta\otimes u^\eta) -\Delta_x \big(\phi(x) (\ue(x)\cdot \vec{n}(x))^2\big)\Bigg) \\
&\quad \quad\quad\quad \quad\quad   \quad\quad\quad\quad   \quad\quad  -2(\nabla_x \phi_i)\cdot \nabla_x \Pe -(\Delta_x  \phi_i) \Pe\Bigg )\,,\label{k0}
\end{aligned}
 \end{equation*}
from which the  estimate (\ref{Piii}) follows. In fact observe that both sides of (\ref{k0}) are smooth functions and that the theorem 3.4.1 of \cite {Krylov} can be applied.
\end{proof}

Next, we turn to estimating  the near the boundary term $\Pbe$.  Once again we observe  from \eqref{regu2} and \eqref{depart}  that $\Pbe$ satisfies the equation
\begin{equation}
\begin{aligned}
 -\Delta\Pbe= \phi_b(x) \Bigg( (\nabla_x\otimes \nabla_x ): (u^\eta\otimes u^\eta)
 & -\Delta_x \big(\phi(x) (\ue(x)\cdot \vec{n}(x))^2\big)\Bigg) \\
 &\quad\quad\quad \quad\quad   -2(\nabla_x \phi_b)\cdot \nabla_x \Pe -(\Delta_x  \phi_b) \Pe \,,
\end{aligned}
\end{equation}
using the fact that $\phi(x)=1$ at the support of $\phi_b$ we obtain
\begin{equation}
\begin{aligned}
 -\Delta\Pbe= \phi_b(x) \Bigg( (\nabla_x\otimes \nabla_x ): (u^\eta\otimes u^\eta) &-\Delta_x \big( (\ue(x)\cdot \vec{n}(x))^2\big)\Bigg) \\&\quad\quad \quad\quad\quad -2(\nabla_x \phi_b)\cdot \nabla_x \Pe -(\Delta_x  \phi_b) \Pe \,. \label{Piiiib}
\end{aligned}
\end{equation}
Establishing estimates for $\Pbe$  involve a more detailed analysis near the boundary for which we will use the explicit form of
 $\big(\nabla\otimes \nabla\big )  :\big(\ue\otimes \ue\big)$ in terms the global geodesic coordinates in $V_\delta \cap \overline{\Omega}\,$. This is the objective of the next:

 \begin{lemma}
 For $x \in V_\delta\cap \overline{\Omega}$ one has:
\begin{equation}\label{geodesicflux}
\begin{aligned}
&(\nabla_x\otimes\nabla_x) :(\ue\otimes \ue)=\nabla_x\cdot \Big (\nabla_x \cdot(\ue\otimes \ue)\Big)=\\
&\frac1{J } \Big(  \del_s\big( J ( \del_s (\ue\cdot \vec n)^2\big) +   2 \del_s \del_{\theta} \big((\ue\cdot \vec n)(\ue\cdot \vec \tau)\big ) + \del_\theta\big(\frac 1J \del_\theta(\ue\cdot \vec \tau)^2\big)\Big)
 +R_b^\eta\,,
 \end{aligned}
\end{equation}
where $R_b^\eta$ involves all the first order derivative terms  and is given by the formula:
\begin{equation}
R_b^\eta= \frac\gamma J \Big(\del_s\big((\ue\cdot \vec n)^2-(\ue\cdot\vec \tau)^2\big)   \Big)+\frac1J\del_\theta\big({\frac\gamma J}(\ue\cdot\vec n)(\ue\cdot\vec \tau)\big)\,.\label{reste}
\end{equation}
\end{lemma}
 \begin{proof}
First observe that for any $C^1$ vector functions $x\mapsto  v(x)$ and $x\mapsto u(x)$  one has the formula
\begin{equation}\label{trivial}
\big (\nabla_x \cdot (\ue\otimes \ue) \big)\cdot   v= \nabla_x\cdot \big((\ue\cdot v ) \ue\big) -(u\otimes u) :\nabla_x v\,.
\end{equation}
Then use this formula with $ v=\vec n $ and $  v=\vec \tau $, respectively, to obtain:
\begin{equation}\label{important}
\begin{aligned}
 &\begin{pmatrix}  \nabla_x\cdot \big((\ue\cdot \vec n) \ue\big) -( \ue\otimes \ue) :\nabla_x \vec n    \\
                             \nabla_x\cdot \big(( \ue\cdot \vec \tau)  \ue\big)  -( \ue\otimes  \ue) :\nabla_x \vec\tau             \end{pmatrix} &&=\begin{pmatrix}   \frac 1 J \big(\del_s(  J (\ue\cdot \vec n)^2 )+   \del_{\theta} (( \ue\cdot \vec n)( \ue\cdot \vec \tau))\big )
  \\
  \frac 1 J\big(\del_s ( J( \ue\cdot\vec \tau) ( \ue\cdot \vec n)) +\del_\theta(( \ue\cdot \vec \tau )^2)\big)
 \end{pmatrix}\\
 &\quad &&   -\begin{pmatrix}\frac\gamma J( \ue\cdot \vec \tau)^2) \\
\frac\gamma J(( \ue\cdot \vec \tau)( \ue\cdot \vec n))  \end{pmatrix}\,.
\end{aligned}
\end{equation}
For the first term of the right-hand side of (\ref{important}) the divergence formula (\ref{clas1}) has been used, while for the second term \eqref{normal-tangent-derivatives} and the gradient formula (\ref{gradients}) have been used. Then once again one uses the divergence formula (\ref{clas1}) to conclude the proof. \end{proof}

Combining the result of the Lemma above with the expression of the Laplacian in geodesic coordinate  (\ref{clas3}), to compute  $\Delta_x \big( (\ue(x)\cdot \vec{n}(x))^2\big)$ and $\Delta_x  \phi_b$ in the right-hand side of \eqref{Piiiib}, equation \eqref{Piiiib} yields the following basic formula for our purpose:
\begin{equation}\label{sanss2}
\begin{aligned}
 -\Delta \Pbe=\frac{\phi_b(s)}J\Bigg( \del_\theta\big(\frac 1J \del_\theta(\ue\cdot \vec \tau)^2 \big)&+2    \del_s \del_{\theta}  \Big((\ue\cdot \vec n)(\ue\cdot \vec \tau) \Big) + J R_b^\eta -\del_\theta\big(\frac 1J \del_\theta(\ue\cdot \vec n)^2 \big)\Bigg )
\\
&-\Big( (\del_s^2\phi_b) \Pe +2(\del_s\phi_b)(\del_s \Pe) +\frac \gamma J P^\eta \del_s \phi_b\Big)\,.
 \end{aligned}
  \end{equation}

  \begin{remark}
  It is important to underline the fact that by the specific choice of $P^\eta$ there are no terms involving the second order derivative with respect to $s$   of $\ue$ and $P^\eta\,$ in the right-hand side of the formula (\ref{sanss2}).
 \end{remark}

The first consequence of  formula (\ref{sanss2}), and the remark above, is the uniform (with respect to $\eta$ and $\alpha$) continuity of the function $\del_s   \Pbe$  which is the objective of the following ``trace":

\begin{proposition}\label{tracelim}
The function $\del_s \Pe_b$ is given by an equation of the following form:
\begin{equation}\label{trace1}
\del_s \Pbe(s,\cdot) =\Lambda^\eta(s,\cdot) + \int_s^{\delta}   \Xi^\eta (s',\cdot)ds'\,,
\end{equation}
with $\Lambda^\eta$ and $\Xi^\eta$ equal to $0$ for $s \ge \delta$ and satisfy the estimates:
\begin{equation}\label{trace2}
\begin{aligned}
&\|\Lambda^\eta\|_{C^{0,\alpha}([0,\delta]; H^{-1}(\R/(L\Z)) )}\le C_b \|(\ue\otimes\ue)\|_{C^{0,\alpha}(\Omega) } + D_b\|\Pe\|_{L^\infty(\Omega)}\,,\\
&\|\Xi^\eta\|_{C^{0,\alpha}([0,\delta]; H^{-2}(\R/(L\Z)) )}\le C_b \|(\ue\otimes\ue)\|_{C^{0,\alpha} (\Omega)} + D_b\|\Pe\|_{L^\infty(\Omega)}\,,
\end{aligned}
\end{equation}
where $C_b, D_b$ and positive constants which are independent of $\eta$ and $\alpha$.
\end{proposition}
\begin{proof}
After using the expression of the Laplacian in geodesic coordinate  (\ref{clas3}) to compute the left-hand side of  (\ref{sanss2}), equation (\ref{sanss2}) gives:
\begin{equation*}
\begin{aligned}
 &-\del_s(J \del_s  \Pbe)=\phi_b(s)\Bigg( \del_\theta\big(\frac 1J \del_\theta(\ue\cdot \vec \tau)^2 \big)+2    \del_s \del_{\theta}  \Big((\ue\cdot \vec n)(\ue\cdot \vec \tau) \Big) + J R_b^\eta -\del_\theta\big(\frac 1J \del_\theta(\ue\cdot \vec n)^2 \big)\Bigg )
\\
&- J\Big( (\del_s^2\phi_b) \Pe +2(\del_s\phi_b)(\del_s \Pe) +\frac \gamma J P^\eta \del_s \phi_b\Big)+  \del_\theta \big( \frac1J\del_\theta P^\eta_b\big)\,,
\end{aligned}
\end{equation*}
where we recall that $R_b^\eta$ is given by \eqref{reste}.
Then multiply this equation by a test function $\Phi(\theta) \in H^2  (\R/(L\Z))$ and integrate once or twice, according to the different terms, with respect to $s $ and $\theta$ to obtain (\ref{trace1}) with  estimates (\ref{trace2}).
\end{proof}
\subsection{$C^{0,\alpha}$ regularity estimate for the boundary layer function $\Pbe\,.$}
To obtain $C^{0,\alpha}$ regularity estimates for $\Pbe$ we decompose it into the sum of two functions
\begin{equation}
\Pbe=\Pbe^b + \Pbe^i\,,
\end{equation}
the first one takes care of the boundary term and the second takes care of the right-hand side of equation (\ref{sanss2}) according to the following formulas (observing that both functions, $\Pbe^b$ and $\Pbe^i$, are identically equal  to $0$ whenever $d(x,\del\Omega) \ge \delta-\epsilon$).

\begin{subequations}
\begin{equation}\begin{aligned}\label{pbb}
 \hbox{ In} \quad V_{\delta}\cap \Omega  \quad -\Delta \Pbe^b =0\,,  \quad& \hbox{on} \quad \del\Omega\quad \del_{\vec n}  \Pbe^b = \gamma (\ue \cdot\vec \tau)^2\,\quad \\
 & \hbox{and on}\quad d(x,\del \Omega) = \delta \quad \Pbe^b= 0\,,
\end{aligned}
\end{equation}
\begin{equation}
\begin{aligned}
&\hbox{in} \quad  V_\delta\cap \Omega
 \, - \Delta \Pbe^i=
 \frac{\phi_b(s)}J\Bigg( \del_\theta\big(\frac 1J \del_\theta(\ue\cdot \vec \tau)^2 \big)+2    \del_s \del_{\theta}  \Big((\ue\cdot \vec n)(\ue\cdot \vec \tau) \Big) + J R_b^\eta
 \\& \hskip 3cm -\del_\theta\big(\frac 1J \del_\theta(\ue\cdot \vec n)^2 \big)
 \Bigg)
-\Big( (\del_s^2\phi_b) \Pe +2(\del_s\phi_b)(\del_s \Pe) +\frac \gamma J P^\eta \del_s \phi_b\Big)\,,
 \\
&  \hbox{on }\quad  \del \Omega \quad\del_{\vec n}\Pbe^i =0 \quad \hbox{while  on} \quad d(x,\del\Omega)=\delta \quad \quad \Pbe^i=0\,. \label{bi}
\end{aligned}
\end{equation}
\end{subequations}
First observe that the function $\Pbe^b$ is a harmonic function satisfying the   homogeneous Dirichlet boundary condition on $d(x,\del\Omega)=\delta$ and the  Neumann boundary condition:
\begin{equation*}
\del_{\vec n} \Pbe^b =\gamma (\ue \cdot\vec \tau)^2 \quad \hbox{on} \quad \del\Omega \,.
\end{equation*}

Therefore, as in the proof of Proposition \ref{Pei}  one has, by elliptic  H\"older regularity theory (cf.  chapter 3 of \cite{LU} chapter 6 or more precisely Theorem 3.4.1 and   Theorem 4.5.1 of \cite {Krylov}), the estimate
\begin{equation}
\|\Pbe^b\|_{C^{0,\alpha}}(\Omega) \le C \|(\ue \otimes \ue)\|_{C^{0,\alpha}(\Omega)}\,.
\end{equation}

Denoting by  $(-\Delta_{dn})^{-1}$ the solution operator of the boundary-value problem (\ref{bi}) which is well defined (due in particular to the homogeneous  Dirichlet boundary condition on $d(x,\del\Omega)=\delta$).  The remaining estimate for $\Pbe^i$ is more subtle. A key point in the proof relies on the fact that   right-hand side  of  equation (\ref{bi}) does not involve any  second order derivative terms with respect to  $s$.

Since the problem is considered in the $V_\delta\cap \Omega$, i.e., in the ``slab"  $(s,\theta) \in (0,\delta))\times (\R/(L\Z))$,  one introduces the Green function, $k$,  associated with $(-\Delta_{dn})^{-1}$ according to the formula:
\begin{equation} \label{Green-Function-bi}
((-\Delta_{dn})^{-1} f)(s,\theta) = \int_{(0,\delta)\times (\R/(L\Z))} k(s,\theta ; s', \theta') f(s',\theta')J(s',\theta')ds'd\theta'\,.
\end{equation}
    Applying the representation  \eqref{Green-Function-bi} to equation \eqref{bi}  one obtains $\Pbe^i$ as the sum of 3 terms:

\begin{equation}\label{Sum-of-three-I}
    \Pbe^i=I_1+ I_2+I_3\,.
\end{equation}

    \begin{equation}
    \begin{aligned}
& I_1= \int_{(0,\delta)\times (\R/(L\Z))}  k(s,\theta; s', \theta')   \phi_b(s')\Bigg( \del_{\theta'}\big(\frac 1J \del_{\theta'}(\ue\cdot \vec \tau)^2 \big)-\del_{\theta'}\big(\frac 1J \del_{\theta'}(\ue\cdot \vec n)^2 \big)+ \\
& \hskip 5cm 2 \del_{s'} \del_{\theta'}  \Big((\ue\cdot \vec n)(\ue\cdot \vec \tau) \Big) \Bigg)(s',\theta') ds'd\theta'\,. \label{I1}
\end{aligned}
    \end{equation}
Integrating  twice with respect to $\theta'$  the first two  terms of the right-hand side of (\ref{I1}) ,  one time with respect to $\theta'$ and one time with respect to $s'$ the third (taking in account the fact that $(\ue\cdot \vec n)(0,\theta')=0$) we obtain:
\begin{equation}\label{I1-new}
\begin{aligned}
&I_1=\int_{(0,\delta)\times (\R/(L\Z))}\del_{\theta'}\Big(\frac1{J(s',\theta')}\del_{\theta'}\big(\phi_b(s') k(s,\theta ; s',\theta')\big )\Big)\Big(\ue\cdot \vec \tau)^2 -(\ue\cdot \vec n)^2 \Big)(s',\theta')ds'd\theta'\\
&+2\int_{(0,\delta)\times(\R/(L\Z))}\del_{s'}\Big(\frac1{J(s',\theta')}\del_{\theta'}\big(\phi_b(s') k(s,\theta ; s',\theta')\big)\Big) \big((\ue\cdot \vec n)(\ue\cdot \vec \tau) \big)(s',\theta') ds'd\theta' \,.
\end{aligned}
\end{equation}
For $I_2$  we write:
\begin{equation*}
\begin{aligned}
&I_2=\int_{(0,\delta)\times  (\R/(L\Z))} k(s,\theta ; s',\theta')\phi_b(s')\gamma (\theta') \Big(\del_{s'}\big((\ue\cdot \vec n)^2-(\ue\cdot\vec \tau)^2\big)\Big)(s',\theta')  ds'd\theta' \\
&+\int_{(0,\delta)\times  (\R/(L\Z))} k(s,\theta ; s',\theta')\phi_b(s')\del_{\theta'}\big({\frac\gamma J}(\ue\cdot\vec n)(\ue\cdot\vec \tau)\big)(s',\theta')ds'd\theta'\,.
\end{aligned}
\end{equation*}
After integration by parts one has
$$
I_2 =:I_2^i + I_2^b
$$
with
\begin{equation}\label{I2i}
\begin{aligned}
&I_2^i=  -\int_{(0,\delta)\times  (\R/(L\Z))} \del_{s'}\Big(\gamma(\theta') \phi_b(s') k(s,\theta ; s',\theta') \Big) \Big( (\ue\cdot \vec n)^2(s',\theta')-(\ue\cdot\vec \tau)^2 (s',\theta')\Big )  ds'd\theta'
\\
&-\int_{(0,\delta)\times  (\R/(L\Z))} \del_{\theta'}\big(k(s,\theta ; s',\theta')\phi_b(s')\big) \big({\frac\gamma J}(\ue\cdot\vec n)(\ue\cdot\vec \tau)\big)(s',\theta')ds'd\theta'
\end{aligned}
\end{equation}
and
\begin{equation}\label{I2b}
 I_2^b=  \int_{ (\R/(L\Z))}  \gamma(\theta')   k(s,\theta ; 0,\theta')    (\ue\cdot\vec \tau)^2 (0,\theta') d\theta'\,.
\end{equation}
Eventually we have:
\begin{equation}\label{I3}
I_3 =- (-\Delta_{dn})^{-1}\Big( (\del_s^2\phi_b) \Pe +2(\del_s\phi_b)(\del_s \Pe) +\frac \gamma J P^\eta \del_s \phi_b\Big)\,.
\end{equation}
Using the classical H\"older theory (cf.  as above chapter 3 of \cite{LU} or   Theorem 6.3.2  of \cite{Krylov} ) we prove the following:
\begin{proposition}
The  terms $I_i \,,$ for  $ i=1,2,3$, in \eqref{Sum-of-three-I} satisfy  an estimate of the form
\begin{equation}
\hbox{for}\quad i=1,2,3\,, \quad \|I_i\|_{C^{0,\alpha}(\Omega)} \le      C\|(u^\eta\otimes u^\eta)\|_{C^{0,\alpha}(\Omega)}+D\|P^\eta\|_{L^\infty(\Omega)}\,.\label{bbasic}
 \end{equation}
 \end{proposition}
 \begin{proof}
    To estimate $I_1$ we use the expression \eqref{I1-new} and follow similar steps for those showing the continuity of the linear operator $(-\Delta_{dn})^{-1}\,,$  defined by the formula (\ref {Green-Function-bi}), as a map from $C^{0,\alpha}$ to $C^{2,\alpha}$ (cf.  once again chapter 3 of \cite{LU}   or  more precisely theorem 6.3.2   of \cite{Krylov} ). Obviously same estimates hold for $I_2^i $ and $I_3$ which involve only first and $0$ order derivatives of the kernel $k(s,\theta ; s',\theta')\,.$  In particular
 for $I_3$ one observes that, in (\ref{I3}),
 $(-\Delta_{dn})^{-1}$    is applied to a function with compact support in $V_\delta\cap \Omega \,.$ Then the term $I_3$  is obtained by convolution with the fundamental
 kernel
 $$\frac1{2\pi} \log\frac{1}{|x|}\,.$$

 Eventually for the term  $I_2^b$, given  by \eqref{I2b},
 one may also directly observe that for  $s'$  close to $0$ and $s-s'$ small,  $k(s,\theta ; s',\theta')$ is given (modulo smooth function) by the formula
 \begin{equation}
  k(s,\theta ; s',\theta')=\frac 1{16\pi}\big( \log(\frac1{ (\theta-\theta')^2+(s-s')^2}) +\log(\frac1{( \theta-\theta')^2+(s+s')^2})\big)
 \end {equation}
 which gives also modulo smooth functions
 $$ k(s,\theta ; 0,\theta')   =\frac 1{4\pi}  \log(\frac1{ (\theta-\theta')^2+s^2})\,.$$
 Hence $I_2^b$ satisfies also estimate (\ref{bbasic}).

  \end{proof}

\subsection{Letting $\eta \rightarrow 0$ and removing the constant $D$\,.}
Since at any point of $\Omega$ the function $P^\eta$ coincides either with the boundary term $P_b^\eta$ or with the interior term $P_i^\eta$ one can collect the estimates from the previous section to write:
\begin{equation}\label{calphapert}
\|P^\eta\|_{C^{0,\alpha}(\Omega)} \le \|P_b^\eta\|_{C^{0,\alpha}(\Omega)} + \|P_i^\eta\|_{C^{0,\alpha}(\Omega)} \le  C\|(u^\eta\otimes u^\eta)\|_{C^{0,\alpha}(\Omega)}+D\|P^\eta\|_{L^\infty(\Omega)}\,,
\end{equation}
where the positive constants $C$ and $D$ are independent of $\eta$. Eventually, we would like to take the limit as $\eta \to 0$. However, we first    state and prove the following:
\begin{proposition}\label{lebeau}
The regularized function $P^\eta$ constructed above satisfy the relation:
\begin{equation}
\|P^\eta\|_{L^\infty(\Omega)}\le C_1   \|(u^\eta\otimes u^\eta)\|_{C^{0,\alpha}(\Omega)}\,, \label{lebeau2}
\end{equation}
with a positive constant $C_1$ which depends on $\alpha$ and $\Omega$, but  is independent of $\eta\,.$
\end{proposition}
\begin{proof}
The proof is done by contradiction.
 Assume that the Proposition  \ref{lebeau} is false. As a result  one can extract a subsequence,   still denoted $P^\eta$, such that:
  \begin{equation}
 \lim_{\eta\rightarrow 0} \frac{ \|(u^\eta\otimes u^\eta)\|_{C^{0,\alpha}(\Omega)}}{\quad \|P ^\eta\|_{L^\infty (\Omega)}} =0\,. \label{zerolim}
 \end{equation}
  Therefore, by \eqref{depart}  the sequence:
  \begin{equation}
 G^\eta = \frac{ P^\eta}{ \quad \|P ^\eta\|_{L^\infty (\Omega)}}
 \end{equation}
solves of the boundary-value problem:
\begin{subequations}
 \begin{equation}
 \hbox{ in } \quad \Omega \quad -\Delta G^\eta = \frac{1}{ \|P ^\eta\|_{L^\infty (\Omega)}}\Big( (\nabla\otimes \nabla) : (\ue\otimes \ue )-\Delta (\phi(x)(\ue\cdot\vec n)^2)\Big)\,, \label{interiorweak}
 \end{equation}
 \begin{equation}
 \hbox{ on } \del  \Omega \quad \del_{\vec n}  G^\eta = \frac{1}{\! \|P ^\eta\|_{L^\infty(\Omega)}}\gamma(\ue\cdot \vec n)^2 \,,
 \label{boundaryweak}
 \end{equation}
 \begin{equation}
 \hbox{and} \quad   \int_\Omega G^\eta (x)dx = \frac{1}{\! \|P ^\eta\|_{L^\infty(\Omega)}}\int_\Omega \phi(x) \big(u^\eta(x)\cdot \vec n(x)\big)^2dx\,.\label{averageweak}
 \end{equation}
 \end{subequations}
 Moreover, as are result of  \eqref{calphapert} and \eqref{zerolim} the sequence $\|G^\eta\|_{C^{0,\alpha}(\Omega)}$ is bounded. Thus, by Arzel\`a-Ascoli theorem $G^\eta$ has subsequence, also denoted by $G^\eta$, which converges strongly to a function $G$ in the $C^{0}(\overline{\Omega})$ norm (in fact it converges in the $C^{0,\beta}(\Omega)$ norm for any $\beta \in (0,\alpha)\,$).
 Obviously $\|G^\eta\|_{L^\infty(\Omega)}=\|G\|_{L^\infty(\Omega)}=1$. Therefore, by \eqref{zerolim}  the right-hand side terms of (\ref{interiorweak}),  \eqref{boundaryweak} and \eqref{averageweak} go to $0$  in the $C^{0,\beta}(\Omega)$ norm, for any $\beta \in (0,\alpha)\,,$ as  $\eta\rightarrow 0$.

Recalling from Proposition \ref{tracelim} the formula
 \begin{equation}\label{trace3}
\del_s \Pe(s,\cdot)=\del_s \Pbe(s,\cdot) =\Lambda^\eta(s,\cdot) + \int_s^{\delta}   \Xi^\eta (s',\cdot)ds'\,, \quad \hbox { for} \quad s\in [0,\delta_3+\epsilon]\,.
\end{equation}
Therefore, from the above and  estimates (\ref{trace2})
we deduce
\begin{equation}
G(s,\theta)-G(0,\theta) = \lim_{\eta\rightarrow 0} \big(G^\eta(s,\theta)-G^\eta(0,\theta)\big) =0\quad \hbox{in} \quad C([0,\delta_3+\epsilon];H^{-2}(\R/(L\Z)))\,,
 \end{equation}
which in particular implies that $\del_s G(0,\cdot)=0$ in $H^{-2}(\R/(L\Z))$. As a result of all the above we conclude that $G$ satisfies the boundary-value problem:
 \begin{equation}\label{G-PDE}
\hbox{ in } \quad \Omega  \quad -\Delta  {G}=0\,,\quad \hbox{ on } \quad \del \Omega \quad \del_{\vec n}  {G}=0\, \quad \hbox{ and  } \quad \int_\Omega {G}(x) dx =0\,.
 \end{equation}
But, ${G}=0\,$ the only solution to \eqref{G-PDE}, which contradicts the fact that $\|G\|_{L^\infty(\Omega)}=1$. This in turn completes the proof.
\end{proof}
 Eventually one observes that by virtue of Lemma \ref{regularized-velocity}  the tensor $(\ue\otimes\ue )$ converges  in the  $C^{0}(\overline{\Omega})$ norm to $(u\otimes u )\,$ (in fact the converges in the $C^{0,\beta}(\Omega)$ norm for any $\beta \in (0,\alpha)\,$), with the uniform estimate $\|\ue\otimes\ue\|_{C^{0,\alpha}(\Omega)} \le C  \|u \otimes u\|_{C^{0,\alpha}(\Omega)}\,,$ where $C$ is independent of $\eta$. By means of  estimates \eqref{calphapert} and (\ref{lebeau2}) one concludes that $\|P^\eta\|_{C^{0,\alpha}(\Omega)}$ is bounded,  hence one can extract a subsequence, also denoted $P^\eta$, which  converges to $P\in C^{0,\alpha}(\Omega)$ in the $C^{0}(\overline{\Omega})$  norm. Moreover, arguing exactly as in the proof of
 Proposition \ref{lebeau} one can show from the above and  equation \eqref{trace3}
\begin{equation}
P(s,\theta)-P(0,\theta) = \lim_{\eta\rightarrow 0} \big(P^\eta(s,\theta)-P^\eta(0,\theta)\big) =  \int_0^s\lim_{\eta \to 0}\Big(\Lambda^\eta(\sigma,\cdot) + \int_{\sigma}^{\delta}   \Xi^\eta (s',\cdot)ds'\Big) d\sigma\,,
 \end{equation}
where the above equality holds in  $C^1([0,\delta_3+\epsilon];H^{-2}(\R/(L\Z))).$ This in particular implies that  $\del_s P(0,\theta) = \gamma(\theta) \big (u(0,\theta)\cdot \vec{\tau} (\theta)\big)^2$ in  $H^{-2}(\R/(L\Z))$. Therefore, from all the above we conclude that
  $P\in C^{0,\alpha}(\Omega)\,,$ is the solution of the boundary-value problem stated in Theorem \ref{Super}, hence the proof of Theorem  \ref{Super} and of  Corollary \ref{super} are completed.

\section{Conclusion and additional remarks}
The above derivation is not a  surprising result, taking into account the present mathematical understanding of fluid mechanics, but it requires a series of technical steps, some of which are inspired by treatment of the problem  in the half space as done in \cite{RRS} and \cite{RRS2}. However   here we  are concerned with a bounded domain with genuinely curved boundary. To address this issue, and for the sake of clarity, we  did focus on the  two-dimensional case  and provide  the most explicit detailed computations. This derivations are based on a global analysis near the boundary and including the interaction between two layers which is also inspired by the recent contribution of Kukavica, Vicol and Wang \cite{KVW}.

We believe that such approach may contribute to extending some of the half space classical results, like the Caflish and Sammartino \cite{CS} stability results for Prandtl equations
in the half space, to more general domains.

From the time of Kolmogorov one knows, as described, for instance, in the book of Frisch  \cite{UF}, that anomalous energy dissipation is genuinely related to the  appearance of turbulence. As it is well known this observation is the origin of a long story in the ``Mathematical Physics" community starting with Onsager  \cite{ON} in 1949, continued in the ``Mathematical community" first by
\cite{CET} and \cite{GEY} with many other contributions later. At present, with results based on the theory of convex integration, as initiated by C. De Lellis and L. Sz\'ekelyhidi Jr.,  one knows  (cf. \cite{BSVVV}, \cite{Isett} and references therein) that $1/3$ is the  critical exponent of the H\"{o}lder  regularity for the absence of anomalous energy dissipation. In particular for any $\alpha<\frac13$ there exist, what are termed as, ``wild but admissible solutions" that do not conserve the energy in the Euler equations (cf. \cite{BSVVV} and \cite{Isett} for the most updated results and references).

 According to physical observations the situation is much more complex in the presence of boundaries and boundary effects. Hence considering sufficient conditions for absence or anomalous dissipation of energy or loss of regularity (these aspects being closely related as shown in the basic article of Kato \cite{Ka}) became recently a subject of attention (cf. in particular \cite{BT2}, \cite{BTW}, \cite{DN},
\cite{ RRS} and \cite{RRS2}.)
As such we argue that the present article may bring some (most probably minor) contributions to the theory of turbulence.

 This is in particular in full agreement with the fantastic vision of Uriel in turbulence. In his book \cite{UF},  Uriel  recognizes the importance of boundary effect in fluid mechanics, very well illustrated with figures (1.4) and (1.11) in  \cite{UF}. In particular,  figure  (1.11)  deals with homogenous turbulence, but such turbulent flow is generated by the boundary effects of the grid. Hence Uriel has also contributed, and subscribed, to the idea that turbulence, anomalous energy dissipation and boundary effects are really closely related.
Therefore we are  honored and very   happy with the opportunity to contribute to a special volume devoted to Uriel with this article as a token of recognition for his friendship and generous contribution to the scientific community. We hope  that the result presented may find its place in this volume devoted to turbulence.

\section{Acknowledgements} We would like to add our warmest thanks to the two anonymous reviewers for their
interest careful, reading and positive comments for this contribution.


\begin{thebibliography}{99}


\bibitem{BT2} C.~Bardos and E.S.~Titi,   Onsager's conjecture for the incompressible
Euler equations in bounded domains. {\it Arch. Ration. Mech. Anal.} {\bf 228(1)} (2018), 197--207.

\bibitem{BTW}  C.~Bardos, E.~Titi, and E.~Wiedemann,  Onsager's conjecture with physical boundaries and an application to the vanishing viscosity limit. {\it Comm. Math. Phys.} {\bf 370(1)} (2019), 291--310.

 \bibitem{BSVVV} T.~Buckmaster, C.~de Lellis, L.~Sz\'ekelyhidi Jr. and V.~Vicol,  Onsager's conjecture for admissible weak solutions. {\it Comm. Pure Appl. Math.} {\bf 72(2)} (2019), 229--274.


\bibitem{CS}  M. Sammartino and R. Caflisch,  Zero viscosity limit for analytic solutions of the Navier-Stokes equation on a half-space. II. Construction of the Navier-Stokes solution. {\it Comm. Math. Phys.} {\bf 192(2)} (1998),  463--491.


  \bibitem{CET} P.~Constantin, W.~E, and E.S.~Titi, Onsager's Conjecture on the energy conservation for solutions of Euler's equation, {\it Comm. Math. Phys.} {\bf 165} (1994), 207--209.


\bibitem{DN} T.~Drivas and H.Q.~Nguyen,   Remarks on the emergence of weak Euler solutions in the vanishing viscosity limit. {\it J. Nonlinear Sci. } {\bf 29(2)}  (2019), 709--721.

\bibitem {GEY} G.L.~Eyink, Energy dissipation without viscosity in ideal hydrodynamics, I. Fourier analysis and local energy transfer, {\it Phys. D}, {\bf 78(3-4)}, (1994), 222--240.

\bibitem{UF}  U.~Frisch,  The Legacy of A.N. Kolmogorov. Cambridge University Press, Cambridge, 1995.



\bibitem{Isett} P.~Isett, A proof of Onsager's conjecture, {\it Ann. of Math.}  \textbf{188(3)} (2018), 1--93.

 \bibitem{Ka} T.~Kato, Remarks on zero viscosity limit for nonstationary Navier-Stokes flows with woundary, {\it Seminar on Nonlinear Partial Differential Equations (Berkeley, Calif., 1983), Math. Sci.
 Res. Inst. Publ.,} Vol. 2, Springer, New York, 1984, pp. 85--98.


\bibitem{Krylov} N.V.~Krylov, Lectures on Elliptic and Parabolic Equations in H\"older Spaces. Graduate Studies in Mathematics, {\bf 12}. American Mathematical Society, 1996.


\bibitem{KVW}  I.~Kukavica,  V.~Vicol and F.~Wang, The inviscid limit for the Navier-Stokes equations with data analytic only near the boundary. {\it  Arch. Ration. Mech. Anal.} {\bf 237} (2020), no. 2,779--827.

 \bibitem{LU} O.~Ladyzhenskaya and N.~Uraltseva, Linear and Quasilinear Elliptic Equations. 	Academic Press, New York, 1968.

\bibitem{ON} L.~Onsager, Statistical hydrodynamics,  {\it Nuovo Cimento} {\bf 6} (1949), 279--287.

 \bibitem{RRS} J.C. Robinson,  J.L.  Rodrigo, W.D. Skipper, Energy conservation for the Euler equations on $\T  \times\R^+$   for weak solutions defined without reference to the pressure. {\it Asymptot. Anal.} {\bf 110} (2018), no. 3-4, 185--202.

 \bibitem{RRS2} J.C. Robinson,  J.L.  Rodrigo, W.D. Skipper ,  Energy conservation in the $3D$  Euler equation on $\T^2 \times\R^+$. Partial Differential Equations in Fluid Mechanics, 224--251, London Math. Soc. Lecture Note Ser., {\bf 452}, Cambridge Univ. Press, Cambridge, 2018.


\end{thebibliography}
\end{document}